\documentclass[11pt]{amsart}
\usepackage{amsthm,amsmath,amssymb}
\usepackage{graphicx}
\usepackage[colorlinks=true,citecolor=black,linkcolor=black,urlcolor=blue]{hyperref}
\usepackage{a4wide}
\usepackage{color, url}
\usepackage{float}
\usepackage{xspace}
\theoremstyle{plain}
\newtheorem{theorem}{Theorem}[section]

\newtheorem{corollary}[theorem]{Corollary}
\newtheorem{propo}[theorem]{Proposition}

\usepackage{bm}
\usepackage{comment}
\usepackage{comment}

\usepackage{gastex}
\usepackage{microtype}

\theoremstyle{definition}

\newcommand{\seqnum}[1]{\href{http://oeis.org/#1}{\underline{#1}}}

\theoremstyle{remark}

\usepackage{enumerate}
\usepackage{url}
\usepackage{amssymb, psfrag,breqn, latexsym}
\usepackage{graphicx,psfrag}
\usepackage{fancyhdr}
\usepackage{diagbox}

\newcommand{\M}{{\mathcal M}}
\newcommand{\D}{{\mathcal D}}

\newcommand{\area}{{\texttt{area}}}

\newcommand{\h}{{\texttt{h}}}
\newcommand{\inter}{{\texttt{inter}}}
\newcommand{\sper}{{\texttt{sper}}}
\newcommand{\up}{{\texttt{up}}}

\newcommand{\Lu}{{\mathcal L}}

\newcommand{\Hi}{{\mathcal H}}
\newcommand{\semip}{semiperimeter\xspace}

\makeatletter
\def\moverlay{\mathpalette\mov@rlay}
\def\mov@rlay#1#2{\leavevmode\vtop{%
   \baselineskip\z@skip \lineskiplimit-\maxdimen
   \ialign{\hfil$\m@th#1##$\hfil\cr#2\crcr}}}
\newcommand{\charfusion}[3][\mathord]{
    #1{\ifx#1\mathop\vphantom{#2}\fi
        \mathpalette\mov@rlay{#2\cr#3}
      }
    \ifx#1\mathop\expandafter\displaylimits\fi}
\makeatother

\title{The Combinatorics of Motzkin Polyominoes}

\date{\today}
\subjclass[2010]{05A15, 05A19}
\keywords{Motzkin word, generating function, bijections.}

\begin{document}

\author[J.-L Baril]{Jean-Luc Baril}
\address{LIB, Universit\'e de Bourgogne Franche-Comt\'e,   B.P. 47 870, 21078, Dijon Cedex, France}
\email{barjl@u-bourgogne.fr}

\author[S. Kirgizov]{Sergey Kirgizov}
\address{LIB, Universit\'e de Bourgogne Franche-Comt\'e,   B.P. 47 870, 21078, Dijon Cedex, France}
\email{sergey.Kirgizov@u-bourgogne.fr}

\author[J. L. Ram\'{\i}rez]{Jos\'e L. Ram\'{\i}rez}
\address{Departamento de Matem\'aticas,  Universidad Nacional de Colombia,  Bogot\'a, Colombia}
\email{jlramirezr@unal.edu.co}

\author[D. Villamizar]{Diego Villamizar}
\address{Escuela de Ciencias Exactas e Ingenier\'ia, Universidad Sergio Arboleda, Bogot\'a, Colombia}
\email{diego.villamizarr@usa.edu.co}

\begin{abstract}
A word $w=w_1\cdots w_n$ over the set of positive integers is a
Motzkin word whenever $w_1=\texttt{1}$, $1\leq w_k\leq w_{k-1}+1$, and
$w_{k-1}\neq w_{k}$ for $k=2, \dots, n$. It can be associated to a
$n$-column Motzkin polyomino whose $i$-th column contains $w_i$
cells, and all columns are bottom-justified.  We reveal bijective
connections between Motzkin paths, restricted Catalan words, primitive
\L{}ukasiewicz paths, and Motzkin polyominoes.  Using the aforementioned
bijections together with classical one-to-one correspondence with Dyck
paths avoiding $UDU$s, we provide generating functions with respect to
the length, area, semiperimeter, value of the last symbol, and number
of interior points of Motzkin
polyominoes. We give asymptotics and closed-form expressions for the total
area, total semiperimeter, sum of the last symbol values, and total
number of interior points over all Motzkin polyominoes of a given
length.  We also present and prove an engaging trinomial relation
concerning the number of cells lying at different levels and first
terms of the expanded $(1+x+x^2)^n$.

\end{abstract}

\maketitle

\section{Introduction}
In the literature,  Dyck and Motzkin paths play a crucial role in many problems from combinatorial theory (see \cite{Krat} and the references wherein).  Remember that  a \emph{Dyck path} of semilength $n$ is a lattice path of  $\mathbb{Z \times Z}$ running from $(0, 0)$ to $(2n, 0)$ that never passes below the $x$-axis and whose permitted steps are $U=(1, 1)$ and  $D=(1,-1)$.  A \emph{Motzkin path} of length $n$  is a lattice path in the first quadrant made of $n$ steps $U$, $D$, and $F=(1,0)$, from the origin to the point $(n,0)$. These paths are enumerated by the well-known Catalan  and Motzkin numbers  (respectively, \seqnum{A000108} and \seqnum{A001006} in the On-line Encyclopedia of Integer Sequences \cite{OEIS}). 

These sequences have a similar behavior since, almost always, they appear together, and  a large number of  similar types of combinatorial classes  are enumerated by these numbers.   For Catalan numbers, we refer to the  catalog of Catalan combinatorial classes compiled by Stanley \cite{stan}.  Similarly, there are several  combinatorial  classes where the  Motzkin family arises. For example, the rooted plane trees are a classical Catalan family, but if we add loops, yields a Motzkin family \cite{Don}. Additional examples of Motzkin families can be found in \cite{Baril, RiordanN} and the references therein. 

In this paper, we focus on the class of Motzkin words introduced by Mansour and Ram\'irez in \cite{ManRamMot}.  Specifically, a word $w=w_1w_2\cdots w_n$ of length $n$ over the set of positive integers is called a \emph{Motzkin word}  whenever $w_1=\texttt{1}$,  $1\leq w_k\leq w_{k-1}+1$, and  $w_{k-1}\neq w_{k}$ for $k=2, \dots, n$. The empty word $\epsilon$ is the only one word of length $0$. For $n\geq 0$, let $\M_n$ denote the set of Motzkin words of length $n$.  For example,
$$\M_5=\{ \texttt{12121}, \  \texttt{12123}, \   \texttt{12312}, \   \texttt{12321}, \   \texttt{12323}, \   \texttt{12341}, \,   \texttt{12342}, \,  \texttt{12343}, \  \texttt{12345}\}.$$ 
In \cite{ManRamMot}, the authors prove that the cardinality of the set $\M_n$ is given by the Motzkin number $m_{n-1}$, where $m_n$ is defined by the  combinatorial sum (see \cite{RiordanN, Don})
$$m_n=\frac{1}{n+1}\sum_{i\geq 0}\binom{n+1}{i}\binom{n+1-i}{i+1}, \quad n\geq 0.$$

The number $m_n$  corresponds to the $n$-th term of the generating function 
$$M(x):=\sum_{n\geq 0}m_nx^n=\frac{1-x-\sqrt{1-2x-3x^2}}{2x^2},$$
and the first few values of the Motzkin numbers are 
$$1, \quad 1,  \quad  2,  \quad  4, \quad  9, \quad  21, \quad  51, \quad  127, \quad  323, \quad  835, \quad  2188, \dots$$
  For all $1\leq k\leq n$, we denote by $\M_{n,k}$ the set of the Motzkin words of length $n$ whose last symbol is  $k$. Let $m(n,k)$ denote the number of Motzkin words in $\M_{n,k}$.   It is clear that $m_n=\sum_{k\geq 1}m(n,k)$. Moreover, in \cite{ManRamMot}  the authors proved that
\begin{align}\label{form}
m(n,k)=\frac{k}{n}\sum_{j=k}^n(-1)^{n-j}\binom nj \binom{2j-k-1}{j-1}.
\end{align}

 A Motzkin word $w=w_1\cdots w_n$ can also be viewed as a polyomino (also called bargraph) whose $i$-th column contains $w_i$ cells for $1\leq i \leq n$, and where all columns are bottom-justified (i.e.~the bottom-most cells of all of its columns are in the same row). The polyomino associated to a Motzkin word of length $n$ is called a \emph{Motzkin polyomino} of length $n$. In Figure~\ref{fig2}, we show all Motzkin polyominoes of length $5$.  
\begin{figure}[H]
\centering
\includegraphics[scale=0.7]{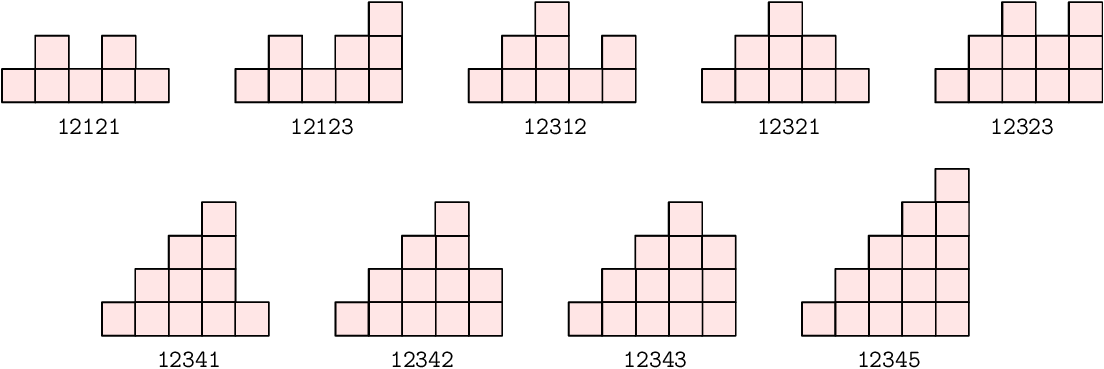}
\caption{Motzkin words  of length 5 and their associated Motzkin  polyominoes.} \label{fig2}
\end{figure}

 In this work, we focus on the class of Motzkin polyominoes, and we investigate enumerative problems about several statistics on these objects.

Let $w$ be a Motzkin word and $P(w)$ its associated polyomino. We denote  by $\area(w)$ the number of cells (or the {\it area}) of $P(w)$, which also is the sum of all $w_i$ for $1\leq i \leq n$. The {\it semiperimeter} of $P(w)$, denoted $\sper(w)$, is half of the {\it perimeter} of $P(w)$, while the perimeter of $P(w)$ is the number of cell borders that do not touch another cell of  $P(w)$. An {\it interior-vertex} of $P(w)$ is a point that belongs to exactly four cells of $P(w)$. We denote by $\inter(w)$ the number of interior points of $P(w)$. For instance,  if $w=\texttt{12341}$, then   $\area(w)=11$, $\sper(w)=9$, and $\inter(w)=3$.  We refer to \cite{Book1} for a historical review on polyominoes, and to  \cite{ManSha2}  for the definitions of many statistical and enumerative methods over polyominoes. Additionally, analogous results are known for bargraphs \cite{Bou, BLE3}, compositions \cite{BleBreKnop}, set partitions \cite{ManA2, ManA3},  Catalan words \cite{CallManRam, Toc, ManRam}, inversion sequences \cite{ArcBleKnop}, and words \cite{ManA, BleBreKnop3}.

\noindent{\bf Motivation and outline of the paper.} The motivation of this work is to provide enumerative results on Motzkin words according to several parameters (area, \semip, and number of interior points) defined on the associated polyominoes. Except for the first section where we exhibit  bijections between Motzkin words and other classical combinatorial classes counted by the Motzkin numbers, all other results are obtained algebraically by considering functional equations for multivariate  generating functions. In Section 2 we give three bijections between Motzkin words and three known combinatorial classes counted by the Motzkin numbers. In Section 3, we focus on the two statistics of the area and the semiperimeter. We provide the trivariate generating function, where the coefficient of $x^np^k q^\ell$ is the number of Motzkin polyominoes $P(w)$ of length $n$, and satisfying $\sper(w)=k$ and $\area(w)=\ell$.
From this, we deduce the generating function for the total sum of the last symbol over all Motzkin words of length $n$, and we give asymptotic approximation for the coefficient of $x^n$. 
In Section 4, we focus on the semiperimeter statistic. We give the generating function for nonempty Motzkin polyominoes with respect to the length and the semiperimeter. We exhibit a bijection between Motzkin polyominoes of length $n$ with semiperimeter $2n-k$, and Motzkin paths of length $n-1$ with exactly $k$ up steps. We deduce a closed-form expression for the total semiperimeter over all Motzkin polyominoes of length $n$. In Section 5, we make a similar study as in Section~4 by considering the area instead of the semiperimeter. Finally, Section~6 is dedicated to the statistic of the number of interior points of the Motzkin polyominoes, and we exhibit a closed-form expression for the total number of interior points over all Motzkin polyominoes of length $n$.

\section{Links between Motzkin words and other Motzkin classes}

Below, we present three bijections between Motzkin words (or equivalently Motzkin polyominoes) and three known combinatorial classes counted by the Motzkin numbers. 

First, let us recursively define  the bijection $\psi$ between $\M_n$ and the set of Motzkin paths of length $n-1$, i.e.~lattice paths in the first quadrant, starting at the origin, ending on the $x$-axis, and made of $(n-1)$ steps $F=(1,0)$, $U=(1,1)$, and $D=(1,-1)$:
$$\psi(w)=\left\{\begin{array}{ll}
\epsilon & \mbox{if } w=\texttt{1}, \\
 F \psi(u)& \mbox{if } w=\texttt{1}(1+u), \mbox{ where } u\neq \epsilon,\\
U \psi(u) D \psi(v) & \mbox{if } w=\texttt{1}(1+u)v, \mbox{ where } u,v\neq \epsilon,
\end{array}\right.
$$
where $u$ and $v$ are Motzkin words and $(1+u)$ corresponds to the word obtained from $u$ by increasing by one all letters of $u$.
For instance, we have 
\begin{multline*}
    \psi(\texttt{12123453412})=U \psi(\texttt{1}) D \psi(\texttt{123453412})=UDU\psi(\texttt{123423})D\psi(\texttt{12})\\=UDUF\psi(\texttt{12312})DF\psi(\texttt{1})=UDUFU\psi(\texttt{12})D\psi(\texttt{12})DF=UDUFUFDFDF.
\end{multline*}
In Figure \ref{MotPath} we show the corresponding Motkzin path.
\begin{figure}[H]
\centering
\includegraphics[scale=0.9]{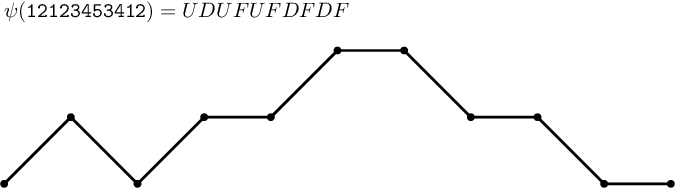}
\caption{Bijection between Motzkin paths and Motzkin words.} \label{MotPath}
\end{figure}

On the other hand, in a previous paper \cite{relation}, the authors
study the avoidance of patterns in Catalan words of length $n$,
i.e.~words $w=w_1w_2\cdots w_n$ satisfying $w_1=\texttt{0}$ and $0\leq w_i\leq
w_{i-1}+1$. They focus on Catalan words of length $n$ avoiding the
pattern $(\geq,\geq)$, that is those that do not contain occurrences
$w_i\geq w_{i+1}\geq w_{i+2}$, or equivalently those avoiding the
consecutive patterns $100$, $000$, $110$, and $210$.  They prove
algebraically that these words are counted by the Motzkin number
$m_n$. Below, we give a bijection $\phi$ establishing a link between
this class of Catalan words and $\M_n$.

Let $w$ be a Catalan word avoiding $(\geq,\geq)$, then $\phi$ is recursively defined  as follows.

$$\phi(w)=\left\{\begin{array}{ll}
\texttt{1} & \mbox{if } w=\epsilon, \\
\texttt{1}(1+\phi(u))& \mbox{if } w=\texttt{1}(1+u),\\
\texttt{1}(1+\phi(\bar{u}))\phi(v)& \mbox{if } w=\texttt{1}(1+u)v, \quad u,v\neq \epsilon,\\
\texttt{1}(1+\phi(u))\texttt{1} & \mbox{if } w=\texttt{11}(1+u), \\
\texttt{1}(1+\phi(\bar{u}))\phi(v)\texttt{1} & \mbox{if } w=\texttt{11}(1+u)v,\quad  u,v\neq \epsilon,
\end{array}\right.
$$
where $u$ and $v$ are some Catalan words, and $\bar{u}$ is obtained from $u$ by deleting the last symbol, and where $(1+u)$ corresponds to the word obtained from $u$ by increasing by one all letters of $u$.
For instance, we have $$\phi(\texttt{1123231231})=\texttt{1}\phi(\texttt{121}) \phi(\texttt{1231})\texttt{1}=\texttt{12323123121}.$$ We leave it as an exercise to check that $\phi$ is a bijection.

\noindent{\bf Link with  primitive \L{}ukasiewicz paths.}
A \emph{\L{}ukasiewicz path} of length $n\geq 0$ is a lattice path in the first quadrant of the $xy$-plane starting at the origin $(0, 0)$,  ending on the $x$-axis, and consisting of $n$ steps lying in the set $\{(1,k): k\geq-1\}$. It is well-known \cite{barpro, Ges, stan} that these paths are enumerated by the Catalan numbers $c_n=\frac{1}{n+1}\binom{2n}{n}$.   A \emph{primitive} \L{}ukasiewicz path is one with exactly one return on the $x$-axis (necessarily at the end). The empty path is not primitive.  Let $\Lu_n$ be the family of primitive \L ukasiewicz paths of length $n$ without horizontal steps.

    There is a bijection between the sets $\M_n$ and $\Lu_{n+1}$. Indeed, for $1\leq i\leq n$ the $i$-th entry of the Motzkin word is exactly the $y$-coordinate  of the final point of the $i$-th step of the corresponding \L ukasiewicz path (see Figure~\ref{Luk}), and the avoidance of horizontal steps becomes from the avoidance of two consecutive identical entries in the Motzkin word.
    
\begin{figure}[H]
\centering
\includegraphics[scale=0.8]{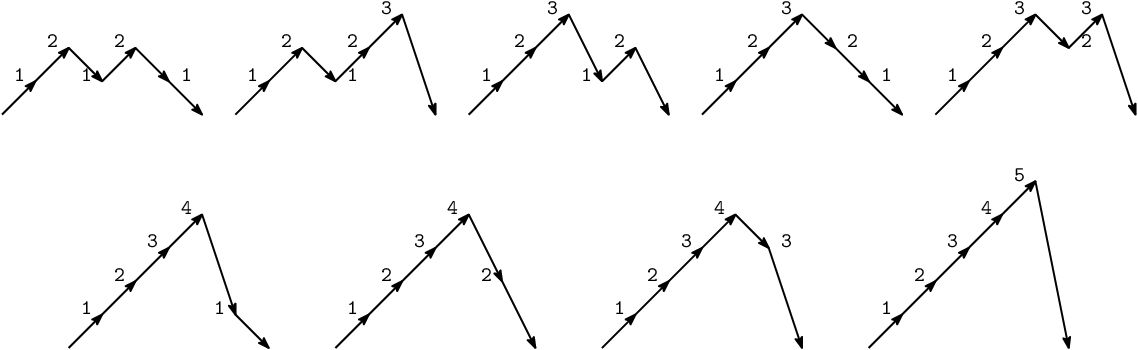}
\caption{Bijection between Motzkin words and primitive \L{}ukasiewicz  paths.} \label{Luk}
\end{figure}

\section{Area and \semip statistics}
In this section, we study the area and \semip statistics on Motzkin words (or equivalently Motzkin polyominoes). 
   Recall that, for all $1\leq i\leq n$, $\M_{n,i}$ is the set of Motzkin words of length $n$ whose last symbol is  $i$. We define the generating functions
$$A_i(x;p,q):=\sum_{n\geq1}x^n\sum_{w\in \M_{n,i}}p^{\sper(w)}q^{\area(w)}.$$
and 
 $$A(x;p,q;v):=\sum_{i\geq1}A_i(x;p,q)v^{i-1}.$$

\begin{theorem}
The generating function $A(x;p,q;v)$ satisfies the functional equation
\begin{align}
A(x;p,q;v)&= p^2qx  + \frac{pqx}{1-qv}A(x;p,q;1) + \left(p^2q^2xv- \frac{pqx}{1-qv}\right)A(x;p,q;qv).\label{eqA3}
\end{align}
\end{theorem} 
\begin{proof}
Let $w$ be a Motzkin word in  $\M_{n}$. We distinguish two cases: ($i$) $w\in \M_{n,1}$ and ($ii$) $w\in \M_{n,i}$ for $i\geq 2$.

Case ($i$). The last symbol of $w$ is $\texttt{1}$. Hence, we have either $w=\texttt{1}$ or $w=w'\texttt{1}$ where $w'\in \M_{n-1,j}$, with $n\geq 2$ and $j\geq 2$. See Figure \ref{deco1} for a graphical representation of these two cases. The generating function for this case is given by 
\begin{align*}
A_1(x;p,q)&=p^2qx+pqx\sum_{j\geq2}A_j(x;p,q).
\end{align*}
From this relation it is clear that
\begin{align}
A_1(x;p,q)&=p^2qx+pqx(A(x;p,q;1)-A_1(x;p,q)).\label{eqA1}
\end{align}

\begin{figure}[H]
\centering
\includegraphics[scale=0.8]{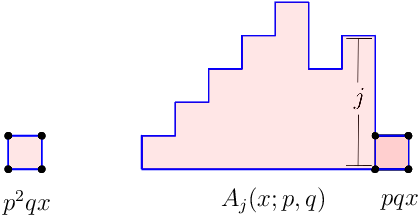}
\caption{Motzkin polyominoes whose last column is of height $1$.} \label{deco1}
\end{figure}

Case ($ii$). The last symbol of $w$ is at least $\texttt{2}$. From the definition of Motzkin word we have the decomposition $w=w'i$, with $w'\in\M_{n-1,j}$, and with either $j=i-1\geq 1$ or $j>i\geq 2$. See Figure \ref{deco2} for a graphical representation of these two cases.

\begin{figure}[H]
\centering
\includegraphics[scale=0.8]{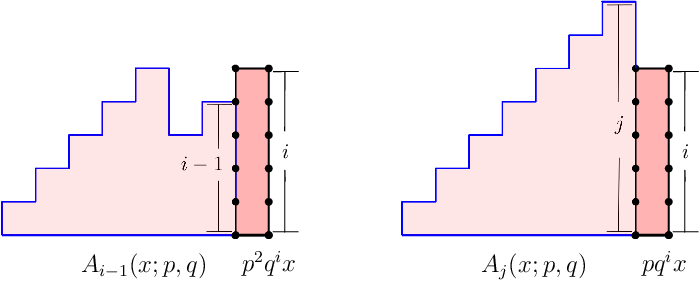}
\caption{Motzkin polyominoes whose last column is of height $\geq 2$.} \label{deco2}
\end{figure}

In terms of generating functions, we deduce  the functional equations:
\begin{align}
A_i(x;p,q)&=p^2q^ixA_{i-1}(x;p,q)+pq^ix\sum_{j\geq i+1}A_j(x;p,q), ~\text{ for }~ i\geq 2.\label{eqA2}
\end{align}

By multiplying equation \eqref{eqA2} by $v^{i-1}$  and summing over $i\geq 2$, we obtain
\begin{align*}
A(x;p,q;v)-A_1(x;p,q)&=\sum_{i\geq 2}p^2q^ixA_{i-1}(x;p,q)v^{i-1} + \sum_{i\geq 2}pq^ix\sum_{j\geq i+1}A_j(x;p,q)v^{i-1}\\
&= p^2q^2xvA(x;p,q;qv)+  px\sum_{j\geq 3}A_j(x;p,q) \sum_{i=2}^{j-1} q^i v^{i-1}\\
&= p^2q^2xvA(x;p,q;qv)+  px\sum_{j\geq 3}A_j(x;p,q)\frac{q^2v-q^jv^{j-1}}{1-qv}\\
&=p^2q^2xvA(x;p,q;qv)+  \frac{pq^2xv}{1-qv}(A(x;p,q;1)-A_1(x;p,q)-A_2(x;p,q)) \\
& \qquad \qquad \qquad - \frac{pqx}{1-qv}(A(x;p,q;qv)-A_1(x;p,q)-qvA_2(x;p,q))\\
&=p^2q^2xvA(x;p,q;qv)+ \frac{pq^2xv}{1-qv}(A(x;p,q;1)-A_1(x;p,q)) \\ 
& \qquad\qquad\qquad  - \frac{pqx}{1-qv}(A(x;p,q;qv)-A_1(x;p,q)).
\end{align*}
Simplifying the last expression using \eqref{eqA1}, we obtain the desired result.
\end{proof}

Note that when $p=q=1$, we have
$$\left(1-xv+\frac{x}{1-v}\right)A(x;1,1;v)=x+\frac{x}{1-v}A(x;1,1;1).$$
The generating function $A(x;1,1;1)$ corresponds to the generating function of the nonempty Motzkin words, so 
$A(x;1,1;1)=xM(x)$. Thus, we deduce the following.

\begin{corollary}   
The generating function for the number of nonempty Motzkin words with respect to the length and the value of the last symbol is 
    \begin{align}
A(x;1,1;v)\cdot v=\frac{xv(1-v) + x^2vM(x)}{1 - v + x - v x + v^2 x}.\label{eqA4}
\end{align}
\end{corollary} 
The first terms of the series expansion are
\begin{multline*}xv +v^2 x^{2}+\left(v^{3}+v\right) x^{3}+\left(v^{4}+2 v^2 +v\right) x^{4}+\left(v^{5}+\bm{3 v^{3}}+2 v^2 +3v\right) x^{5}+\\
\left(v^{6}+4 v^{4}+3 v^{3}+7 v^2 +6v\right) x^{6}+
\left(v^{7}+5 v^{5}+4 v^{4}+12 v^{3}+14 v^2 +15v\right) x^{7}+ O(x^8).
\end{multline*}
There are three polyominoes of length $5$ ending with a column of height three (see Figure~\ref{fig2}). By differentiating $A(x;1,1;v)\cdot v$ at $v=1$, we deduce:
\begin{corollary} \label{cor33} The generating function for the total sum $g_n$ of the last symbol in all Motzkin words of length $n$ is 
\begin{align*}
\frac{1 - x - 2 x^2-\sqrt{1 - 2 x - 3 x^2}}{2 x^2}.
\end{align*}
An asymptotic approximation for the coefficient $g_n$ of $x^n$ is given by
$$\frac{3 \sqrt{3}\, \left(\frac{1}{n}\right)^{\frac{3}{2}} 3^n}{2 \sqrt{\pi}}.$$
\end{corollary}
The asymptotic approximation is easily obtained with a singularity analysis (see \cite{Fla,Orl}).

The first terms of $g_n$, $1\leq n\leq 10$, are 
$$1,\quad 2, \quad 4, \quad 9, \quad 21, \quad 51, \quad 127, \quad 323, \quad 835, \quad 2188, \dots.$$
This sequence  corresponds to the Motzkin numbers \seqnum{A001006} in the OEIS \cite{OEIS}. Finally, the expected value of the last symbol is $m_n/m_{n-1}$, and an asymptotic   is $3$.

Notice that we can retrieve  Corollary~\ref{cor33} as follows. From a Motzkin word $w\in\mathcal{M}_{n-1,k}$, we can construct $k$ Motzkin words $w'=wx\in\mathcal{M}_{n,x}$ by adding a letter $x\in[1,k-1]\backslash \{k\}$. Conversely, any Motzkin word of length $n$ can be uniquely constructed with this process. This means that $\sum_{k=1}^{n-1} km(n-1,k)=m_{n}$ which gives a simple combinatorial interpretation of Corollary~\ref{cor33}.

\section{The \semip statistic}
In this section we study the \semip statistic on Motzkin polyominoes. By \eqref{eqA3} with $q=1$, we obtain
\begin{align}\label{eq6}
\left(1-p^2xv+\frac{px}{1-v}\right)A(x;p,1;v)&= p^2x  + \frac{px}{1-v}A(x;p,1;1).
\end{align}

In order to compute $S(x,p):=A(x;p,1;1)$, we use the kernel method~ (see \cite{Kernel,  pro}). 
The  method consists in cancelling the coefficient  of $A(x;p,1;v)$
by taking  the small root $v_0$ (the ones going to $0$ for $x\sim 0$) of $1-p^2xv_0+\frac{px}{1-v_0}$, namely
$$v_0=\frac{1 + p^2 x-\sqrt{1 - 2 p^2 x - 4 p^3 x^2 + p^4 x^2}}{2p^2x}.$$
When the factor on the left-hand side equals zero (and thus satisfies the equation), the right-hand side also equals zero and so  the solution $v_0$ can be used on the right-hand side to obtain $A(x;p,1,1)$ (see Theorem \ref{mth1}).

\begin{theorem}\label{mth1}
The generating function for the number of nonempty Motzkin polyominoes  according to the length and the \semip is given by
\begin{align*}
S(x,p)=&\frac{1 - p^2 x-\sqrt{1 - 2 p^2 x - 4 p^3 x^2 + p^4 x^2}}{2px}.
\end{align*}
\end{theorem}
The first terms of the series expansion of $S(x,p)$ are $$p^2 x + p^4 x^2 + (p^5 + p^6) x^3 + (3 p^7 + p^8) x^4 + (2 p^8 + 
    \bm{6 p^9} + p^{10}) x^5 + O(x^6).$$
Figure~\ref{fig3} yields the $6$ Motzkin polyominoes of length $5$ and semiperimeter $9$.
\begin{figure}[H]
\centering
\includegraphics[scale=0.7]{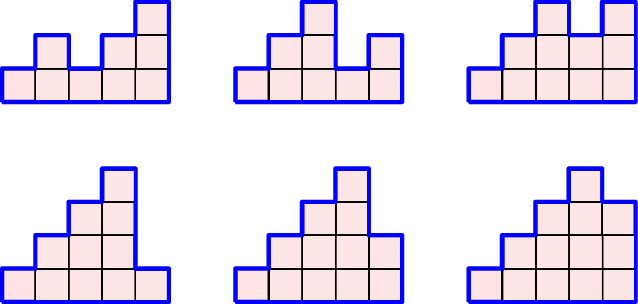}
\caption{Motzkin polyominoes of length $5$ and semiperimeter $9$.} \label{fig3}
\end{figure}

\begin{theorem}\label{bijectionT}
The map $\psi$ (defined in Section 2) induces a bijection between the Motzkin polyominoes (or equivalently words) in $\M_n$ of \semip $2n-k$ for $k=1, 2, \dots, \lfloor\frac{n-1}{2}\rfloor$, and the Motzkin paths  of length $n-1$ with exactly $k$ up steps.
\end{theorem}
\begin{proof} Let us prove the statement by induction on $n$. The case $n=0$ is trivial since $\psi(\texttt{1})=\epsilon$. Let us assume the statement is true for $m\leq n$, and let us prove it for $n+1$. Let $w\in\M_{n+1}$ with \semip $2(n+1)-k$. We denote by  $\up(M)$ the number of up-steps of the Motzkin path $M$.

If $w= \texttt{1}(1 + u), u\neq \epsilon$, then  we have  $\sper(u)=2n-k$,  $\psi(w)=F\psi(u)$, and $\up(\psi(w))=\up(\psi(u))$. Using the induction hypothesis on $u$, we have  $\up(\psi(u))=k$, and then $\up(\psi(w))=k$ as expected.

If $w=\texttt{1}(1 + u)v, u,v\neq \epsilon$, then  we have  $\sper(u)=s_1$ and $\sper(v)=s_2$ with $s_1+s_2+1=2(n+1)-k$,  $\psi(w)=U\psi(u)D\psi(v)$, and $\up(\psi(w))=1+\up(\psi(u))+\up(\psi(v))$. Using the induction hypothesis on $u$ and $v$, we have  $\up(\psi(u))=2|u|-s_1$, and $\up(\psi(v))=2|v|-s_2$, where $|u|$ is the length of $u$. Hence, this implies that $\up(\psi(w))=1+2|u|-s_1+2|v|-s_2=2+2|u|+2|v|-2(n+1)+k=k$, as expected.

Considering these two cases, the induction is completed.
\end{proof}

Using relation (\ref{eq6}), we deduce the following.
\begin{theorem}\label{mthh}
The generating function for the number of nonempty Motzkin polyominoes  according to the length, the \semip, and the value of the last letter is given by
\begin{align*}
A(x;p,1;v)=\frac{1 + p^2 (x - 2 v x)-\sqrt{1 - 2 p^2 x - 4 p^3 x^2 + p^4 x^2}}{2 (1 + p x + p^2 v^2 x - v (1 + p^2 x))}.
\end{align*}
\end{theorem}

The first terms of the series expansion of $A(x;p,1;v)$ are  \begin{multline*}
p^2 x + p^4 v x^2 + (p^5 + p^6 v^2) x^3 + (p^7 + 2 p^7 v + 
    p^8 v^3) x^4 + \\ (2 p^8 + \bm{p^9 + 2 p^9 v + 3 p^9 v^2} + p^{10}v^4) x^5 + O(x^6).
\end{multline*}
From Figure \ref{fig3} we can verify that there are 1, 2, and 3 Motzkin polyominoes of length $5$ and semiperimeter 9, whose value of the last symbol (height of last column) is 1, 2, and 3, respectively. 

Let $s(n,i)$ denote the sum of the semiperimeters of all  Motzkin polyominoes  of length $n$ whose last column has height $i$.  The first few values are
$$[s(n,i)]_{n, i\geq 1}=
\left(
\begin{array}{ccccccc}
 2 & 0 & 0 & 0 & 0 & 0 &\cdots\\
 0 & 4 & 0 & 0 & 0 & 0 &\cdots\\
 5 & 0 & 6 & 0 & 0 & 0 &\cdots\\
 7 & 14 & 0 & 8 & 0 & 0 &\cdots\\
 25 & 18 & 27 & 0 & 10 & 0 &\cdots\\
 61 & 72 & 33 & 44 & 0 & 12 &\cdots\\
 \vdots  &  \vdots& \vdots & \vdots & \vdots & \vdots&\ddots \\
\end{array}
\right).$$
Notice that from the decomposition given in Figure \ref{deco2},  we have for  $n\geq 2$ and $2\leq i\leq n$,
$$s(n,i)=s(n-1,i-1) + 2m(n-1,i-1) + \sum_{j=i+1}^{n-1}\left(s(n-1,j) + m(n-1,j)\right),$$
where $m(n,i)$ is given in \eqref{form}. If we consider the difference $s(n,i)-s(n,i-1)$, then for  $n\geq 2$ and $3\leq i\leq n$, we obtain the recurrence relation
\begin{multline*}
s(n,i)=s(n,i-1)+s(n-1,i-1)-s(n-1,i-2)-s(n-1,i)\\
+2(m(n-1,i-1)-m(n-1,i-2))-m(n-1,i).
\end{multline*}

Let $s(n)$ be the total \semip  over all Motzkin polyominoes of length $n$.  By Corollary~\ref{mthh} at $v=1$,  we deduce:
\begin{corollary}
    The generating function of the sequence $s(n)$ is given by
\begin{align*}
\frac{1 + x^2-(1+x)\sqrt{1 - 2 x - 3 x^2}}{2x\sqrt{1 - 2 x - 3 x^2}}.
\end{align*}
An asymptotic approximation of $s(n)$ is given by $$\frac{5 \sqrt{3}\, \sqrt{\frac{1}{n}}\, 3^n}{6 \sqrt{\pi}}.
$$
\end{corollary}
The first few values for $1\leq n\leq 10$ are
$$2, \quad 4, \quad 11, \quad 29, \quad 80, \quad 222, \quad 624, \quad 1766, \quad 5030, \quad 14396, \dots $$
This sequence does not appear in  the OEIS.  The expected value of the semiperimeter is $5n/3$. In Corollary \ref{cor1} we give an explicit relation for the sequence $s(n)$ using the \emph{central trinomial coefficient} $T_n$, that is the coefficient of $x^n$ in the expansion $(x^2 + x + 1)^n$. From the multinomial theorem we see that
$$T_n=\sum_{k=0}^n\binom{n}{k}\binom{n-k}{k}.$$
Moreover,  the generating function of the central trinomial coefficients is given by
 \begin{align}\label{trino} 
  T(x):=\sum_{n\geq0} T_n x^n=\frac{1}{\sqrt{1-2x-3x^2}}.
  \end{align}

\begin{corollary}\label{cor1}
The total \semip over all Motzkin polyominoes of length $n$ is given by
$$s(n)=T_n + 2T_{n-1} - m_{n-1},$$
where $m_n$ is the $n$-th Motzkin number.
\end{corollary}
\begin{proof}
By differentiating $S(x,p)$ at $p=1$, Theorem \ref{mth1} gives
\begin{align*}
\frac{\partial}{\partial p}S(x,p)\rvert_{p=1}&=\frac{1 + x^2-(1+x)\sqrt{1 - 2 x - 3 x^2}}{2 x\sqrt{1 - 2 x - 3 x^2}}\\
&=-1 + T(x) + 2x T(x) - x M(x).
\end{align*}
Comparing the $n$-th coefficient we obtain the desired result.
\end{proof}

We leave as an open question to find a combinatorial interpretation of the equation obtained in Corollary~\ref{cor1}.

\section{The area statistic}
The goal of this section is to analyze the area statistic on Motzkin polyominoes. By \eqref{eqA3} with $p=1$ we obtain
\begin{align}\label{areaeq}
A(x;1,q;v)&= qx  + \frac{qx}{1-qv}A(x;1,q;1) -\frac{qx(1 - q v + q^2 v^2)}{1-qv}A(x;1,q;qv).
\end{align}
By iterating this equation an infinite number of times (here we assume $|x|<1$ or $|q|<1$), we obtain
\begin{align*}
A(x;1,q;v)=\sum_{j\geq1}(-1)^{j-1}q^{j}x^j\left(1+\frac{A(x;1,q;1)}{1-q^jv}\right)
\prod_{i=1}^{j-1}\frac{1-q^i v + q^{2i}v^2}{1-q^iv}.
\end{align*}
By setting $v=1$, and solving for $A(x;1,q;1)$, we can state the following result.
\begin{theorem}
The generating function $U(x,q):=A(x;1,q;1)$ for the number of nonempty Motzkin polyominoes according to the length and the area is given by
\begin{align*}
U(x,q)=\frac{\sum_{j\geq1}(-1)^{j-1}q^{j}x^j\prod_{i=1}^{j-1}\frac{1-q^i+q^{2i}}{1-q^i}}
{1-\sum_{j\geq1}(-1)^{j-1}\frac{q^{j}x^j}{1-q^j}\prod_{i=1}^{j-1}\frac{1-q^i+q^{2i}}{1-q^i}}.
\end{align*}
\end{theorem}
The first terms of the series expansion of $U(x,q)$ are 
\begin{align*} q x + q^3 x^2 + (q^4 + q^6) x^3 &+ (q^6 + q^7 + q^8 + 
    q^{10}) x^4+\\
    &+ (\bm{q^7 + 3 q^9 + 2 q^{11} + q^{12} + q^{13} + q^{15}}) x^5 + O(x^6).
\end{align*}
    We refer to Figure~\ref{fig4} for an illustration of the polyominoes of length $5$.
    
\begin{figure}[H]
\centering
\includegraphics[scale=0.7]{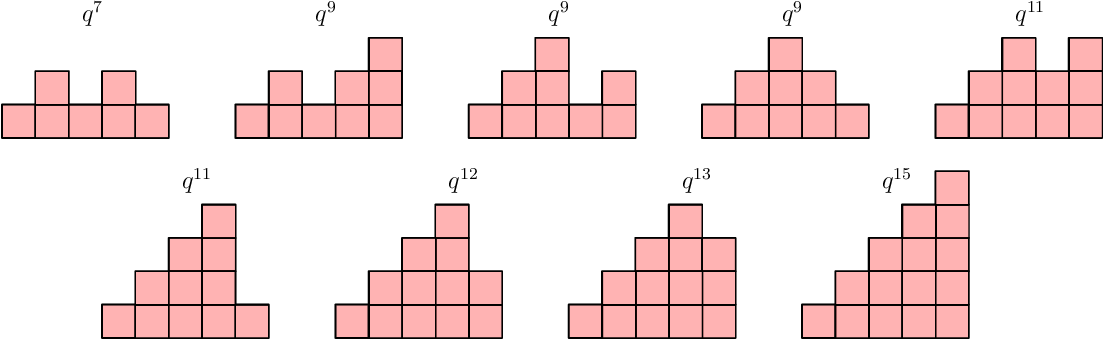}
\caption{Motzkin  polyominoes of length $5$ and their weighted area.} \label{fig4}
\end{figure}

\begin{theorem}\label{contifrac}
    The generating function for the number of Motzkin polyominoes, including
the empty word, according to the length and area is given by the infinite continued
fraction
\begin{align*}
1+U(x,q)&=\cfrac{1}{1-\cfrac{qx}{1+qx-\cfrac{(1+qx)q^2x}{1+q^2x-\cfrac{(1+q^2x)q^3x}{\ddots}}}}.
\end{align*}
\end{theorem}
\begin{proof}
First note that the Motzkin words of length $n$ are in bijection with Dyck paths of length $2n$ avoiding the consecutive subword $UDU$. Indeed,  given a Dyck path avoiding $UDU$, we associate a Motzkin word formed  by the $y$-coordinate of each final point of the up steps.  For example, in Figure \ref{figdyckword}  we show the Dyck path associated to the Motzkin word $\texttt{123212343}\in \M_{9}$. We denote by $\D$ the set of Dyck paths avoiding $UDU$.

\begin{figure}[H]
\centering
\includegraphics[scale=0.8]{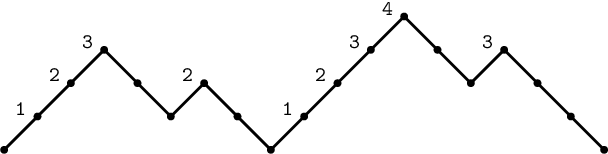}
\caption{Dyck path associated to the Motzkin word $\texttt{123212343}$.} \label{figdyckword}
\end{figure}
From this bijection, it is clear that the area of a Motzkin polyomino is equivalent to count the area of a Dyck path avoiding  the subword $UDU$ (the area of a Dyck path is the area of the region bounded by the path and the $x$-axis). Given a Dyck path $P$ in $\D$  from the first return decomposition  it can be one of the following options: $\epsilon, UD, UP'DP''$, where $P', P''\in \D$ and $P'\neq \epsilon$. Let $J(x,q)$ be the bivariate generating function for the Dyck paths in $\D$ according to the length and the area. From the decomposition we obtain  the functional equation 
$$J(x,q)=1+xq+xq(J(qx,q)-1)J(x,q).$$
Therefore, we deduce
$$J(x,q)=\frac{1+xq}{1-xq(J(qx,q)-1)}.$$
Iterating this expression we obtain the desired result. 
\end{proof}

Let $u(n,i)$ denote the total area of the  Motzkin polyominoes  of length $n$ that end with a column of height  $i$.  The first few values are
$$[u(n,i)]_{n, i\geq 1}=
\left(
\begin{array}{ccccccc}
 1 & 0 & 0 & 0 & 0 & 0&\cdots \\
 0 & 3 & 0 & 0 & 0 & 0 &\cdots\\
 4 & 0 & 6 & 0 & 0 & 0 &\cdots\\
 7 & 14 & 0 & 10 & 0 & 0&\cdots \\
 27 & 21 & 33 & 0 & 15 & 0&\cdots \\
 75 & 89 & 45 & 64 & 0 & 21 &\cdots\\
 \vdots  &  \vdots & \vdots &  \vdots &  \vdots & \vdots &\ddots\\
\end{array}
\right).$$

From the decomposition given in Figure \ref{deco2},  we have for  $n\geq 2$ and $2\leq i\leq n$,
$$u(n,i)=u(n-1,i-1) + i\cdot m(n-1,i-1) + \sum_{j=i+1}^{n-1}\left(u(n-1,j) + i\cdot m(n-1,j)\right),$$
where $m(n,i)$ is given in \eqref{form}. If we consider the difference $u(n,i)-u(n,i-1)$, then for  $n\geq 2$ and $3\leq i\leq n$, we obtain the recurrence relation
\begin{multline*}
u(n,i)=u(n,i-1)+u(n-1,i-1)-u(n-1,i)-u(n-1,i-2)\\+i\cdot m(n-1,i-1) -(i-1)m(n-1,i-2)-i\cdot m(n-1,i) 
+\sum_{j=i}^{n-1}m(n-1,j).
\end{multline*}

Let $u(n)$ be the total area of over all Motzkin polyominoes of length $n$.  The first few values are
$$1, \quad 3, \quad 10, \quad 31, \quad 96, \quad 294, \quad 897, \quad 2727, \quad 8272, \quad 25048, \dots$$
This sequence corresponds to the sequence  \seqnum{A055217} in the OEIS \cite{OEIS}. Recently, Goy and Shattuck \cite{GS} gave a combinatorial  interpretation for this sequence by using   marked Motzkin paths. Our interpretation is different and  probably new.  In the following theorem we give an interesting combinatorial formula to calculate the sequence $u(n)$.

\begin{theorem}\label{areacom}
The total area over all Motzkin polyominoes of length $n$ is given by
$$u(n)=\frac{1}{2}\left(3^n-T_n\right)=\frac{1}{2}\left(3^n-\sum_{k=0}^n\binom{n}{k}\binom{n-k}{k}\right),$$ and an asymptotic is $3^n/2$.
\end{theorem}
\begin{proof}
Let
$B(x;v)=\frac{\partial}{\partial q}A(x;1,q;v)\mid_{q=1}$. Then by differentiating \eqref{areaeq} with respect to $q$, we obtain
\begin{align*}
B(x;v)&=x+\frac{x^2M(x)}{(1-v)^2}+\frac{x}{1-v}B(x;1)\\
&-\frac{x(1-2v+4v^2-2v^3)}{(1-v)^2}A(x;1,1;v)-\frac{x(1-v+v^2)}{1-v}(B(x;v)+v\frac{\partial}{\partial v}A(x;1,1;v)).
\end{align*}
By using \eqref{eqA4}, we obtain
\begin{align}\label{dereqA4}
\frac{\partial}{\partial v}A(x;1,1;v)&=\frac{\partial}{\partial v}\left(\frac{x(1-v) + x^2M(x)}{1 - v + x - v x + v^2 x}\right)\\
&=\frac{1 - 2 v x - x^2 - 2 v x^2 + 2 v^2 x^2 - (1+x-2vx)\sqrt{1 - 2 x - 3 x^2}}{2 (1 - v + x - v x + v^2 x)^2}.
\end{align}
Therefore
\begin{multline*}
\frac{(1 - v + x - v x + v^2 x)^3}{1-v}B(x;v)=x(1 - v + x - v x + v^2 x)^2+\frac{x^2(1 - v + x - v x + v^2 x)^2M(x)}{(1-v)^2} \\
 +\frac{x(1 - v + x - v x + v^2 x)^2}{1-v}B(x;1) \\ -\frac{x(1-2v+4v^2-2v^3)(1 - v + x - v x + v^2 x)}{(1-v)^2}\left(x(1-v) + x^2M(x)\right)\\
 -\frac{xv(1-v+v^2)}{2(1-v)}\left(1 - 2 v x - x^2 - 2 v x^2 + 2 v^2 x^2 - (1+x-2vx)\sqrt{1 - 2 x - 3 x^2}\right).
\end{multline*}
By twice differentiating this equation with respect to $v$ and taking $v=xM(x)+1$, we obtain that 
\begin{align*}
B(x,1)=\frac{1 + x -\sqrt{1 - 2 x - 3 x^2}}{2 - 4 x - 6 x^2}=\frac{1}{2}\left(\frac{1}{1-3x} - \frac{1}{\sqrt{1 - 2 x - 3 x^2}}\right).
\end{align*}
Comparing the $n$-th coefficient we obtain the desired result.
\end{proof}

\subsection{Link with trinomial coefficients}

The sequence $u(n)$ (\seqnum{A055217}) corresponds to the sum of the first $n$
coefficients of $(1+x+x^2)^n$.  Enigmatically, the monomials from the
first part of the expansion of $(1+x+x^2)^n$ can be literally written on
the cells of all Motzkin polyominoes of size $n$ in a simple one-to-one manner.
Term $x^k$ goes onto a cell of height $n-k$,
see Figure~\ref{1_x_x2}. This
subsection is devoted to explain this fact, which is easy to state,
but not easy to prove.  Theorem~\ref{areacom} follows from the results
of this section if we replace $x$ by 1, consider the symmetry of
coefficients of $(1+x+x^2)^n$, and subtract the central trinomial coefficient.

\begin{figure}[ht]
  \includegraphics[width=36em]{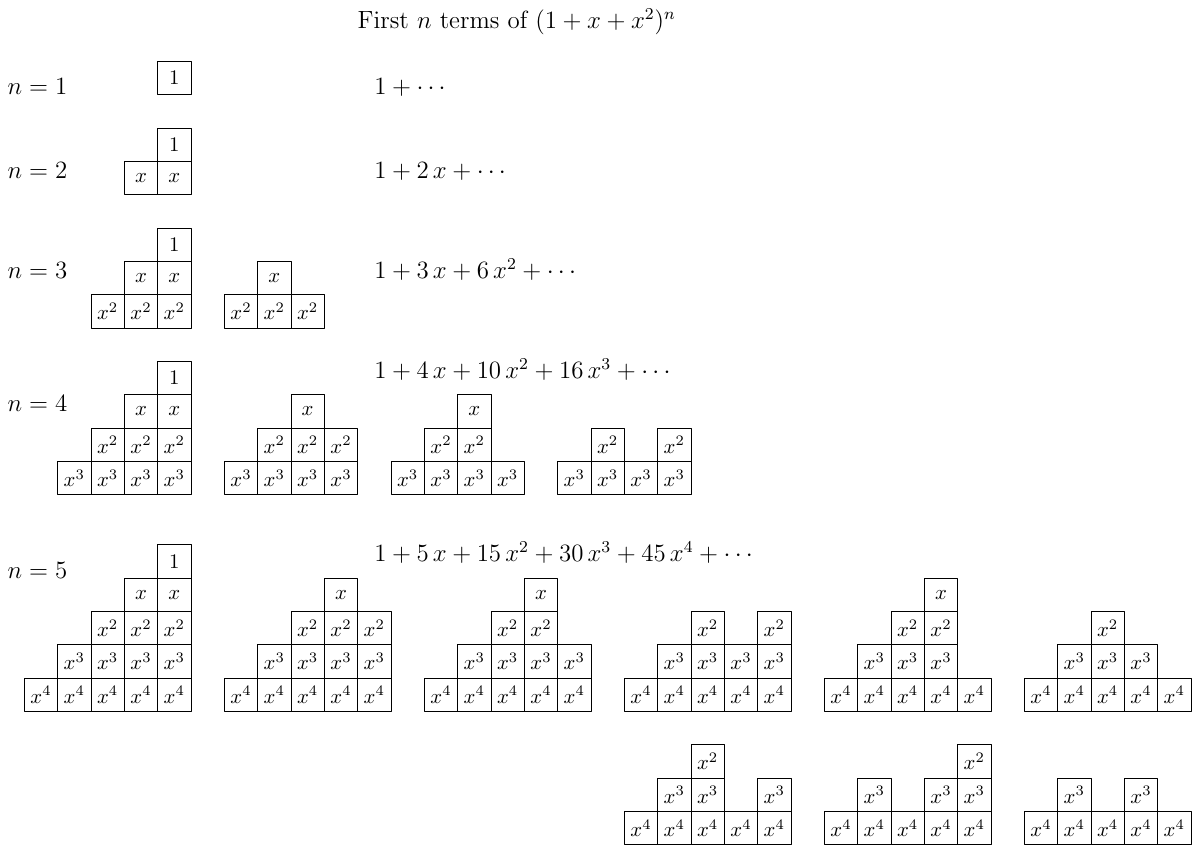}
  \caption{Correspondence between Motzkin polyominoes of size $n$ and first $n$ terms of the expansion of $(1+x+x^2)^n$.}
  \label{1_x_x2}
\end{figure}

Let $T(n,i)$ be the $i$-th coefficient in the expansion of $(1+x+x^2)^n$.  It is not difficult to prove that 
\begin{align}\label{recTrinomial}
T(n,i)=T(n-1,i)+T(n-1,i-1) + T(n-1,i-2), \quad 0 \leq i \leq n-2,
\end{align}
and $T(n,n-1)=T_{n-1}+T(n-1,n-2) + T(n-1,n-3)$, where $T_n$ is the central trinomial coefficient, that is $T_n=T(n,n)$.  We also have $T(n,n-1) = n |\M_n| = n m_{n-1}$, see~\seqnum{A005717} in the OEIS for more information.
If $i>n-1$ or $i<0$ we define $T(n,i)=0$, moreover $T(1,1)=1$.  The first few values of this array are
$$[T(n,i)]_{n\geq 1,i\geq 0}=
\left(
\begin{array}{ccccccccc}
 1 & 0 & 0 & 0 & 0 & 0 & 0 & 0 &\cdots\\
 1 & 2 & 0 & 0 & 0 & 0 & 0 & 0 &\cdots\\
 1 & 3 & 6 & 0 & 0 & 0 & 0 & 0 &\cdots\\
 1 & 4 & 10 & 16 & 0 & 0 & 0 & 0 &\cdots\\
 1 & 5 & 15 & 30 & 45 & 0 & 0 & 0 &\cdots\\
 1 & 6 & 21 & 50 & 90 & 126 & 0 & 0&\cdots \\
 1 & 7 & 28 & 77 & 161 & 266 & 357 & 0 &\cdots\\
 1 & 8 & 36 & 112 & 266 & 504 & 784 & 1016 &\cdots\\
  \vdots &\vdots & \vdots & \vdots & \vdots & \vdots & \vdots & \vdots &\ddots\\
\end{array}
\right).
$$

Let $h(n,i)$ be the total number of cells of height $i$ in $\M_n$.
From Figure~\ref{1_x_x2}, we have 
$h(5,1)=45, h(5,2)=30, h(5, 3)=15, h(5,4)=5, h(5,5)=1$. Notice that these numbers are related to the first 5 coefficients of  the expansion of 
$$(1+x+x^2)^5=\bm{\textcolor{blue}{1 + 5 x + 15 x^2 + 30 x^3 + 45 x^4}} + 51 x^5 + 45 x^6 + 30 x^7 + 
 15 x^8 + 5 x^9 + x^{10}.$$

Let $w$ be a Motzkin word. We denote by $\h_i(w)$ the number of cells of height $i$ in the Motzkin polyomino associated with $w$. We introduce the following  generating functions
$$H_i(x,q):=1+\sum_{w\in\M}x^{|w|}q^{\h_i(w)}$$
and
$$B_i(x):=\left.\frac{\partial H_i(x,q)}{\partial q}\right|_{q=1}.$$
From the definition it is clear that $[x^n]B_i(x)=h(n,i)$.

\begin{theorem}\label{teoSer1}
 For $i\geq 2$, we have 
 $$H_i(x,q)=\frac{1+x}{1-(H_{i-1}(x,q)-x)},$$
 and $H_1(x,q)=1+qxM(xq)$,  where $M(x)$ is the generating function of the Motzkin numbers.
\end{theorem}
\begin{proof} Motzkin words of length $n$ are in bijection with Dyck paths of length $2n$ avoiding the subword $UDU$ (the $i$-th entry of the word $w$ corresponds to the $y$-coordinate of the endpoint of the $i$-th up-step of its corresponding Dyck path). From this bijection it is clear that 
$$H_1(x,q)=1 + xq + xq(H_1(x,q)-1)H_1(x,q).$$
Solving this equation, we obtain that $H_1(x,q)=1+qxM(xq)$. Analogously, from the decomposition given in the proof of Theorem  \ref{contifrac} we obtain the functional equation
$$H_i(x,q)=1+x+x(H_{i-1}(x,q)-1)H_i(x,q), \quad i\geq 2. \qedhere$$ 
\end{proof}

For example, for $i=3$ we have the expression 
\begin{multline*}
    H_3(x,q)=\frac{2 - x - q (2 + x^2)+x\sqrt{(1 + q x) (1 - 3 q x)}}{2 (1 - q - x + q (1 - x) x)}\\=1 + x + x^2 + (1 + q) x^3 + (1 + 2 q + q^2) x^4 + (\bm{1 + 3 q + 3 q^2 + 
    2 q^3}) x^5 +O(x^6).
    \end{multline*}

Figure~\ref{Ser} yields the  Motzkin polyominoes of length $5$ and the cells of height 3.

\begin{figure}[H]
\centering
\includegraphics[scale=0.7]{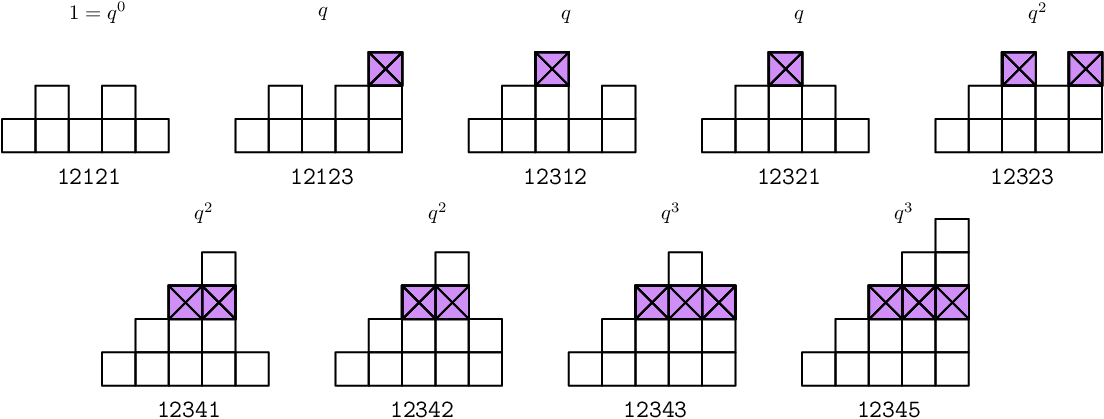}
\caption{Motzkin polyominoes of length $5$ and the cells of height 3.} \label{Ser}
\end{figure}

\begin{propo}
 For $i\geq 1$, we have 
 $$B_i(x)=\frac{x^iM^i(x)}{1-x-2x^2M(x)},$$
 where $M(x)$ is the generating function of the Motzkin numbers.
 \label{prop:bm}
\end{propo}
\begin{proof}
 We proceed by induction on $i$. The identity  clearly holds for $i = 1$; see Theorem \ref{teoSer1}. Now
suppose that the result is true for $i-1$. We prove it for $i$:
\begin{align*}
 B_i(x)&=\left.\frac{\partial H_i(x,q)}{\partial q}\right|_{q=1}=\left.\frac{\partial}{\partial q}\left(\frac{1+x}{1-(H_{i-1}(x,q)-x)}\right)\right|_{q=1}=\left.\frac{\partial}{\partial q}\left(\frac{1}{1-\frac{x}{1-x}H_{i-1}(x,q)}\right)\right|_{q=1}\\
 &=\frac{\frac{x}{1+x}\frac{x^{i-1}M^{i-1}(x)}{1-x-2x^2M(x)}}{\left(1-\frac{x}{1+x}(1+xM(x)))\right)^2}=\frac{(1+x)x^iM^{i-1}(x)}{(1-x^2M(x))^2}=\frac{x^iM^i(x)}{1-x-2x^2M(x)}.
\end{align*}
\end{proof}

A \emph{Motzkin walk} of length $n$ is a Motzkin path prefix without the condition that never passes below the $x$-axis (see for instance \cite{BanFla}). A Motzkin walk ending on $x$-axis is also known as \emph{Grand Motzkin path}. Let $g(n,i)$ be the number of  Motzkin walks of length $n$ ending at height $i$. It is clear that $g(n,i)=g(n-1,i-1)+g(n-1,i) + g(n-1,i+1)$. We define the generating function
$$G_i(x)=\sum_{n\geq 0}g(n,i)x^n.$$
\begin{theorem}\label{rechh}
    For all $i\geq 0$, $G_i(x)=B_i(x)$. Moreover, 
    For $n > 0, i > 1$, we have
    $$
    h(n,i) = h(n-1,i-1) +
    h(n-1,i) + h(n-1,i+1),
    $$
$h(n,1)  = n \cdot [x^{n-1}] M(x)$.   
\end{theorem}
\begin{proof}
    Let $P$ be a Motzkin walk ending at height $i$. It can be decomposed as $GUM_1UM_2U\cdots U M_i$, where $G$ is a Motzkin walk ending at height zero and $M_j$ is a Motzkin path for all $1\leq j \leq i$. Therefore, we have the generating function
    $G_i(x)=x^iG_0(x)M^i(x)$, where $M(x)$ is the generating function of the Motzkin numbers. Notice that $G_0(x)=1+xG_0(x)+2x^2M(x)G_0(x)$, because a nonempty Motzkin walk ending at height zero is either $FG'_0$, or $UMDG'_0$, or $D\bar{M}UG'_0$ where $G'_0$ is a Motzkin walk ending at height zero, $M$ is a Motzkin path (ending on the $x$-axis) and $\bar{M}$ is a Motzkin walk obtained from a Motzkin path (ending on the $x$-axis) after exchanging $U$ and $D$. So, we have $G_0(x)=1/(1-x-2x^2M(x))$. Therefore, 
    $$G_i(x)=\frac{x^iM^i(x)}{1-x-2x^2M(x)}=B_i(x).$$
  Now,  it is clear that the sequences $h(n,i)$ and $g(n,i)$ are the same, therefore $h(n,i)$ satisfies the desired recurrence relation. 
\end{proof}

From Theorem \ref{rechh} and \eqref{recTrinomial}, it is possible 
to verify  that the sequence $T(n,i)$ and $h(n,n-i)$ satisfy the same recurrence relation with the same initial values, therefore they are the same.

\begin{theorem}
    The number of cells of height $n-i$ ($0\leq i<n$) in all Motzkin polyominoes of length $n$ is given by the trinomial coefficient $T(n,i)$. 
\end{theorem}

Now we can give an alternative proof of the combinatorial identity given in the Theorem \ref{areacom}. 

\emph{Second proof of Theorem \ref{areacom}}. Notice that  $\sum_{i=0}^{n-1}T(n,i)=u(n)$.
But the trinomial coefficients are palindromic, that is $T(n,i)=T(n,2n-i)$, then
$$\sum_{i=0}^{2n}T(n,i)=2u(n)- T(n,n)=2u(n)-T_n.$$
Notice that we have to subtract the central trinomial coefficient.  But it is clear that this sum is $3^n$, therefore $3^n=2u(n)-T_n$.

We end this section by asking the following (open) questions: can we obtain an interpretation moving down a level to binomial coefficients? Can we obtain a generalization of this result by using   tetranomial coefficients (i.e.~coefficients of the polynomials $(1+x+x^2+x^3)^n$), and for which  polyomino classes?

\section{The interior points statistic}

In this section, we study the statistic of the number of interior points on Motzkin polyominoes. As in the previous section, for all $1\leq i\leq n$, $\M_{n,i}$ is the set of the Motzkin words of length $n$ whose last symbol is  $i$, and we define the generating functions
$$A_i(x;q):=\sum_{n\geq1}x^n\sum_{w\in \M_{n,i}}q^{\inter(w)}.$$
and 
 $$A(x;q;v):=\sum_{i\geq1}A_i(x;q)v^{i-1}.$$

\begin{theorem}\label{intP}
The generating function $A(x;q;v)$ is given by 
\begin{align*}
A(x;q;v)=\sum_{j\geq1}x^j\left(1+A(x;q;1)\frac{1}{1-q^jv}\right)
\prod_{i=1}^{j-1}\left(x q^{i-1}v - \frac{x}{1-q^iv}\right).
\end{align*}
\end{theorem} 
\begin{proof} According to the decomposition given in Figures~\ref{deco1} and \ref{deco2}, we obtain the relations
\begin{align*}
  A_1(x;q)&= x + x  \sum_{j\geq 2}A_j(x;q),\\
  A_i(x;q)&=xq^{i-2}A_{i-1}(x;q) + x q^{i-1}\sum_{j\geq i+1}A_j(x;q).
\end{align*}

By multiplying the last equation by $v^{i-1}$   and summing over $i\geq 2$, we obtain the functional equation
\begin{multline*}
    A(x;q;v)-A_1(x;q)=\frac{xqv}{1-qv}A(x;q;1) + \left(xv - \frac{x}{1-qv}\right)A(x;q;qv)\\-\left(\frac{xqv}{1-qv}- \frac{x}{1-qv}\right)A_1(x;q).
\end{multline*}
Simplifying this expression we obtain the equation
\begin{align}\label{pointeqec}
    A(x;q;v)=x+\frac{x}{1-qv}A(x;q;1) + \left(xv - \frac{x}{1-qv}\right)A(x;q;qv).
\end{align}
By iterating the last equation an infinite number of times (here we assume $|x|<1$ or $|q|<1$), we obtain the desired result.
\end{proof}

By setting $v=1$ in Theorem \ref{intP}, and  solving for $A(x;q;1)$ we can state the following result.
\begin{corollary}\label{pointeq}
The generating function $H(x,q):=A(x;q;1)$ for the number of nonempty Motzkin polyominoes according to the length and the number of interior points is given by
\begin{align*}
H(x,q)&=\frac{\sum_{j\geq1}x^j\prod_{i=1}^{j-1}\left(q^{i-1}-\frac{1}{1-q^i}\right)}
{1-\sum_{j\geq1}\frac{x^j}{1-q^j}\prod_{i=1}^{j-1}\left(q^{i-1}-\frac{1}{1-q^i}\right)}.
\end{align*}
\end{corollary}

The first terms of the series expansion of $H(x,q)$ are 
\begin{multline*} H(x,q)=x + x^2 + (1 + q) x^3 + (1 + q + q^2 + q^3) x^4 \\ + (\bm{1 + 2 q + q^2 + 
    2 q^3 + q^4 + q^5 + q^6}) x^5+O(x^6).
\end{multline*}

  We refer to Figure~\ref{figintp} for an illustration of the polyominoes of length $5$.
    
\begin{figure}[H]
\centering
\includegraphics[scale=0.7]{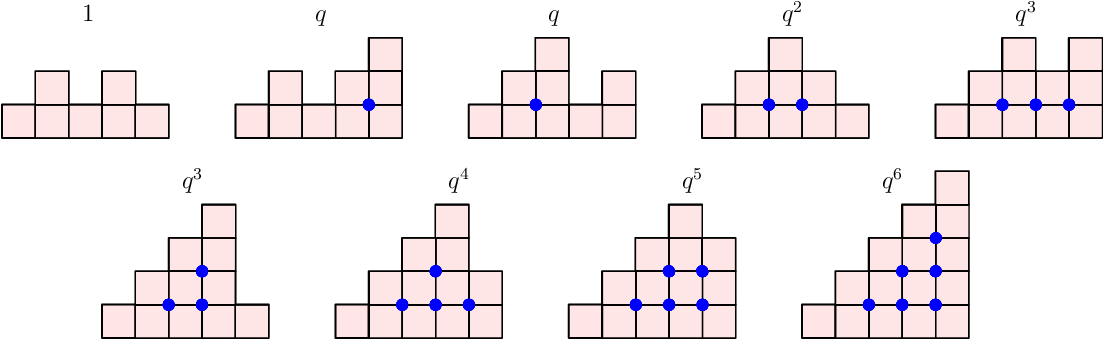}
\caption{Motzkin  polyominoes of length $5$ and their weighted interior points.} \label{figintp}
\end{figure}

\begin{corollary}
The generating function of the number of interior points over all Motzkin polyominoes of length $n$ is 
$$\frac{2 - 3 x - 5 x^2 -(2-x-2x^2)\sqrt{1 - 2 x - 3 x^2}}{2 x (1 + x) (1 - 3 x)},$$ and an asymptotic for the $n$-th coefficient of the series expansion is $3^n/2$. The expected value of the number of interior points is $\sqrt{\frac{\pi}{3}}n^{3/2}$.
\end{corollary}
\begin{proof}
Let
$B(x;v)=\frac{\partial}{\partial q}A(x;q;v)\mid_{q=1}$. Then by differentiating \eqref{pointeqec} with respect to $q$, we obtain
\begin{align*}
B(x;v)&=\frac{vx^2M(x)}{(1-v)^2} + \frac{x}{1-v}B(x;1)\\
&-\frac{xv}{(1-v)^2}A(x;1,v)+\left(xv-\frac{x}{1-v}\right)(B(x;v)+v\frac{\partial}{\partial v}A(x;1;v)).
\end{align*}
From  \eqref{dereqA4}, we know that
\begin{align*}
\frac{\partial}{\partial v}A(x;1;v)=\frac{1 - 2 v x - x^2 - 2 v x^2 + 2 v^2 x^2 - (1+x-2vx)\sqrt{1 - 2 x - 3 x^2}}{2 (1 - v + x - v x + v^2 x)^2}.
\end{align*}
Therefore
\begin{multline*}
\frac{(1 - v + x - v x + v^2 x)^3}{1-v}B(x;v)=\frac{x(1 - v + x - v x + v^2 x)^2}{1-v}B(x;1) \\+ \frac{v}{2} \left(1 - x - v x - 3 x^2 + 2 v x^2 - v^2 x^2 - x^3 + 3 v x^3 - 
   3 v^2 x^3 + 2 v^3 x^3 \right. \\ \left.- (1-xv-x^2+vx^2-v^2x^2)\sqrt{1 - 2 x - 3 x^2}\right).
\end{multline*}
By twice differentiating this equation with respect to $v$ and taking $v=xM(x)+1$, we obtain that 
\begin{align*}
B(x,1)=\frac{2 - 3 x - 5 x^2 -(2-x-2x^2)\sqrt{1 - 2 x - 3 x^2}}{2 x (1 + x) (1 - 3 x)}.
\end{align*}
Comparing the $n$-th coefficient we obtain the desired result.
\end{proof}

As an  application of Pick's theorem (cf. \cite[pp. 40]{Beck}), one may establish a relation between the area, semiperimeter, and number of interior points of a Motzkin polyomino. Pick's theorem says that the area of a simple polygon with integer vertex coordinates is equal to  $I+B/2-1$, where $I$ is the number of interior points and $B$ is the number of points lying on the boundary. Let $\texttt{int}(n)$ be the sum of the interior points  over all Motzkin polyominoes of length $n$.

\begin{corollary}
The number of interior points  over all Motzkin polyominoes of length $n$ is 
$$\frac{1}{2}\left(3^n-3T_n\right)-2T_{n-1}+2m_{n-1}.$$
\end{corollary}
\begin{proof}
 From Pick's theorem we can obtain the relation
 $$u(n)=\texttt{int}(n) + s(n)-m_{n-1}.$$
Now the identity follows  from Corollary  \ref{cor1} and Theorem \ref{areacom}. 
\end{proof}

The first few values of the sequence $\texttt{int}(n)$ for $n\geq 3$ are
$$1,\quad 6, \quad 25, \quad 93, \quad 324, \quad 1088, \quad 3565, \quad 11487, \dots$$

This sequence does not appear in the OEIS.

\textbf{Acknowledgments.} The authors are grateful to the referees for the detailed comments and corrections that helped us improve the paper. José L. Ramírez was partially supported by Universidad Nacional de Colombia. The third author would like to thank the  Laboratoire d’Informatique de Bourgogne for the warm hospitality during his visit at Universit\'e de Bourgogne where part of this work was done. 
This work was supported in part by ANR-22-CE48-0002 funded by l’Agence Nationale de la Recherche and by the project ANER ARTICO funded by
Bourgogne-Franche-Comté region (France).

\end{document}